\documentclass{birkjour}

\usepackage{hyperref}

\let\oldtocsection=\tocsection

\let\oldtocsubsection=\tocsubsection

\let\oldtocsubsubsection=\tocsubsubsection

\renewcommand{\tocsection}[2]{\hspace{0em}\oldtocsection{#1}{#2}}
\renewcommand{\tocsubsection}[2]{\hspace{1em}\oldtocsubsection{#1}{#2}}
\renewcommand{\tocsubsubsection}[2]{\hspace{2em}\oldtocsubsubsection{#1}{#2}}

\providecommand{\R}{\mathbb{R}}

\providecommand{\leq}{\leqslant}
\providecommand{\geq}{\geqslant}

\providecommand{\eps}{\varepsilon}

\renewcommand{\div}{\operatorname{div}}
\newcommand{\curl}{\operatorname{curl}}

\newcommand{\dist}{\operatorname{dist}}

\newtheorem{Theorem}{Theorem}

\newtheorem{Lemma}{Lemma}
\newtheorem{Remark}{Remark}

\begin{document}

\date{\today}
\title[A Kato type criterion for vortex sheets] {A Kato type criterion for the zero viscosity limit of the incompressible Navier-Stokes flows with vortex sheets data}

\author{Franck Sueur\footnote{Institut de Math\'ematiques de Bordeaux, 
UMR CNRS 5251, Universit\'e de Bordeaux, Franck.Sueur@math.u-bordeaux.fr}}

\maketitle

\begin{abstract}
There are a few examples of solutions to the incompressible  Euler equations
which are piecewise smooth with a discontinuity of the tangential velocity across a hypersurface evolving in time: the so-called vortex sheets. 
An important open problem is to determine whether or not these  solutions  can be obtained as  zero viscosity limits of  the incompressible Navier-Stokes solutions in the energy space. 
In this  paper we establish a couple of sufficient conditions  similar to the one obtained by  Kato  in   [T.~Kato. \textit{Remarks on zero viscosity limit for nonstationary Navier-Stokes flows with boundary.} Seminar on nonlinear partial differential equations, 85-98, Math. Sci. Res. Inst. Publ., 2,  1984] 
 for  the convergence of Leray solutions  to  the Navier-Stokes equations in a bounded domain with no-slip condition
toward smooth solutions to the Euler equation.
\ \par \ 

2010 Mathematics Subject Classification: 35B25, 35B30, 76D05, 76D10. \ \par \ 
Key words and phrases: Vanishing viscosity, vortex sheets, boundary layer theory.
\end{abstract}

\section{Introduction}
\label{sec-intro} 
In fluid mechanics a vortex sheet is a hypersurface across which the tangential component of the flow velocity is discontinuous while the normal component is continuous. 
Because of the discontinuity in the tangential velocity the vorticity is infinite on the hypersurface hence the terminology. 
Examples of solutions to the incompressible Euler equations for which an initial vortex sheet evolves in a smooth hypersurface for positive times are rare because of the Kelvin-Helmholtz instability. 
Let us mention the examples  provided by \cite{dr,ssbf} for analytic data and by \cite{BardosTiti-PP,DM}  for plane-parallel flows. 
Let us also refer here  to the survey  \cite{BL}  for more on vortex sheets.

An important open problem is to determine whether or not these  solutions  can be obtained as  zero viscosity limits of solutions to the  incompressible  Navier-Stokes equations.
 In particular, Bardos, Titi and Wiedemann have shown in  \cite{BardosTiti-Wiede} that the zero-viscosity limit can serve as a selection principle for the Euler equations in the case  of 
 initial data for which there exist non-unique weak solutions to the incompressible Euler equations satisfying the weak energy inequality,  including a vortex sheet solution with plane-parallel symmetry. 
 They prove that  the latter is the  zero-viscosity limit, in a weak sense, of any sequence of Leray-Hopf solutions to the Navier-Stokes solutions.
 
In general the difficulty to establish the zero-viscosity limit towards vortex sheets 
 is  that the  fluid tangential velocity  has $O(1)$ variation
in a layer containing the hypersurface, similarly to the  boundary layer associated with  the no-slip condition on a fixed wall. 

For the latter a result by Kato   \cite{Tosio} establishes  the convergence of  Leray solutions to the Navier-Stokes equations 
to a smooth Euler solution  in the energy space provided that the energy dissipation rate of the viscous flow in
a boundary layer of width proportional to the viscosity vanishes, with an appropriate condition for the initial data.  
Since then, this result was extended in various ways, see \cite{BardosTiti-turbu,CEIV,CKV,Kelliher-2007,Kelliher-2008,Kelliher-2017,Wang}. Let us also mention here the recent survey \cite{Maekawa-Mazzucato} for more about  the zero viscosity limit of the incompressible Navier-Stokes flows. 

In this paper we extend Kato's result to the case of a vortex sheet by establishing a couple of sufficient conditions for  the convergence of Leray solutions to the Navier-Stokes in the full space to a vortex sheet. 
These two conditions: Condition \eqref{KatoCondition} and Condition \eqref{KatoConditionJump} below, both involve $L^2$ norms of derivatives of the fluid velocity  on a boundary layer of width proportional to the viscosity. 
Indeed  Condition \eqref{KatoCondition}  involves 
the energy dissipation rate of the viscous flow and is therefore very similar to Kato's  condition.  
 On the other hand Condition \eqref{KatoConditionJump} involves the difference of derivatives of the fluid velocity between one side of the hypersurface and the other.

 \section{Setting}
\label{sec-setting} 

Let $d=2$ or $3$. 
We consider  $T>0$  and $\Sigma$  a smooth compact connected hypersurface of $[0,T] \times \R^d$,  given as the zero level set of a signed smooth function $\varphi(t,x)$, such that, in a small neighborhood of $\Sigma$, 
\begin{equation}
\label{cladis} 
{|\varphi(t,x)|} =  \mathrm{dist}(x, \Sigma_t),
\end{equation}
where,  for every time $t$ in $[0,T]$, we denote by $\Sigma_t \subset  \R^d$ the projection of  $\Sigma$ on  $\{ t\} \times \R^d$.
We assume that the two connected components 
 $\Omega_{t,\pm}$ of $\R^d \setminus \Sigma_t $ are given by 
 $$\Omega_{t,\pm}  := \{ x \in  \R^d \ / \,  \pm \varphi(t,x) > 0\} .$$

We denote by $L^2_\sigma {(\R^d)} $  the closure in $L^2{(\R^d)}$ of the space $ \mathcal{C}^\infty_\sigma ([0,T); {\R^d})$ of smooth divergence free 
vector fields and we will use the notation $    \mathcal{C}_w([0,T];  L^2_\sigma {(\R^d)} ) $ for vector fields depending on time continuously on $[0,T]$ 
with respect to the weak topology of $L^2 {(\R^d)} $.

 Let    $u^E$ in  $\mathcal{C}_w([0,T];  L^2_\sigma {(\R^d)} ) $ such that for every time $t$ in $[0,T]$, the restrictions of $ u^E_\pm  (t,\cdot)$ to  $\Omega_{t,\pm}$
 admit some smooth extensions  to $\overline{\Omega_{t,\pm}} $, with  traces  $ u^E_\pm  (t,\cdot)$  on   $\Sigma_t $  satisfying 
\begin{equation}
\label{eikonal}
\partial_t \varphi  +  u^E_+ \cdot \nabla \varphi   = \partial_t \varphi  + u^E_- \cdot \nabla \varphi = 0 .
\end{equation}
Let us precise that in \eqref{eikonal} the notation $\nabla$ refers to the gradient with respect to the space variables only, and it is the same in the sequel. 

We assume that there exists a scalar function $p^E$ which is,  for every time $t$ in $[0,T]$,   smooth in  $\Omega_{t,\pm}$ up to the boundary and continuous at $\Sigma_t$, such that, in $\Omega_{t,\pm}$, 
$$\partial_t  u^E +  \div ( u^E \otimes u^E) + \nabla p^E= 0 .$$

Then for every time $t$ in $[0,T]$, 
\begin{equation} \label{EulerBodyStrongEnergy}
       \| u^E (t) \|_{L^2 (\R^d )} =  \| u^E_0 \|_{L^2 (\R^d )} .
\end{equation}
We say that $u^E$ is a \textit{vortex sheet associated with $\Sigma$}. Observe in particular that  the tangential component of $u^E$ can be  discontinuous across $\Sigma$ while the normal component is continuous,  as a consequence of \eqref{cladis}  and \eqref{eikonal}.
Moreover  $u^E$  is a weak solution to the incompressible Euler equations in $\R^d$.
\begin{Remark}
Observe that we do not use any subscript $E$ for $\Sigma$ and  $\varphi$; the reason is that  there is no counterpart for the Navier-Stokes equations in the sequel so that there is no ambiguity: $\Sigma$ and  $\varphi$ are always associated with the Euler solution. 
\end{Remark}

A natural question is  whether or not a vortex sheet is a limit when  $\eps \rightarrow 0^+$ 
of solutions $u^{\eps} $ to the  incompressible  Navier-Stokes equations. 
\begin{eqnarray}
\label{NSI}
\partial_t  u^{\eps}  +  \div (u^{\eps}  \otimes u^{\eps} ) + \nabla  p^{\eps} = \eps \Delta u^{\eps}  ,\ \div u^{\eps}  = 0 . 
\end{eqnarray}
Here we will consider weak  solution to~\eqref{NSI} in the sense of Leray. 
We  recall  that, for an initial data $u_0$ in  $L^2_\sigma {(\R^d)} $, 
$$u^{\eps}\in  \mathcal{C}_w([0,T];  L^2_\sigma {(\R^d)} ) \cap L^2((0,T); H^1(\R^d))$$
 is a weak Leray
solution to~\eqref{NSI}  associated with   $u_0$ 
if it satisfies, for any $\phi $ in $\mathcal{C}^\infty_\sigma ([0,T], {\R^d})$, 
\begin{gather}
\label{wNS}
     \int_{\R^d}  u_0 \cdot \phi (0,\cdot) \,  dx dt 
     -     \int_{\R^d}  u^{\eps} (T,\cdot) \cdot \phi (T,\cdot) \,  dx dt 
     +   \int_0^T \,   \int_{\R^d}  u^{\eps}  \cdot \partial_t \phi\,  dx dt 
    \\ \nonumber     +  \int_0^T \,   \int_{\R^d}  (u^{\eps}  \cdot \nabla  \phi)  \cdot u^{\eps}  \,  dx dt 
         -  \eps  \int_0^T \,    \int_{\R^d} \nabla u^{\eps}  : \nabla\phi \,  dx dt 
         = 0 ,
\end{gather} 
 and the strong energy inequality: for almost every $0 \leq \tau < t \leq T$, 
\begin{equation} \label{NSBodyWeakEnergy}
    \begin{split}
     \frac12   \| u^{\eps}  (t) \|_{L^2{(\R^d)}}^2
        +   \eps \iint_{(\tau, t) \times \R^d} | \nabla u^{\eps} |^2 \,  dx dt 
         \leq   \frac12 \| u^{\eps} (\tau) \|_{L^2{(\R^d)}}^2 .
    \end{split}
\end{equation}
%
%

\section{Main result}
\label{sec-MR} 
In this section we state the main result of this paper: Theorem \ref{ThKato}  below.  First we introduce some notations which appear in the statement of this result. 
Let us recall that we consider  $d=2$ or $3$,  $T>0$  and $\Sigma$  a smooth compact connected hypersurface of $(0,T) \times \R^d$.
For any $c > 0$,  for every time $t$ in $[0,T]$,  we set  
\begin{equation*}
\mathcal{V}_{t,c} := \{  x \in  \R^d / \dist (x,\Sigma_t ) <  c \}  .
\end{equation*}
For every $t$ in $[0,T]$, for $c> 0$ small enough, 
the reflexion across $\Sigma_t$ is well-defined on the set of the functions $f$ whose restrictions to ${\Omega_{t,\pm}} $
admit smooth extensions  to $\overline{\Omega_{t,\pm}} $. 
This map associates with $f$ a function $\tilde{f}$ defined as follows.
If for some time $t$, we  denote by  $s$ the tangential coordinates so that $(s,\varphi)$ are local coordinates  then  the function $\tilde{f}$ is given explicitly by $\tilde{f} (t,s,\varphi) = f (t,s,-\varphi)$.

With $f$  we associate  the function 
\begin{equation}
\label{reflexion}
[f]  := f  - \tilde{f}, 
\end{equation}
which is loosely speaking, at $(t,x)$,  the jump of $f$ across $\Sigma_t$ at distance $2{|\varphi(t,x)|}$.

We can now state the main result of the paper. 

\begin{Theorem}
\label{ThKato}
Let $d=2$ or $3$,  $T>0$  and $\Sigma$  a smooth compact connected hypersurface of $(0,T) \times \R^d$.
Let $u^E$  a vortex sheet associated with $\Sigma $.
Let $(u^{\eps}_{0} )$ a family, indexed by $\eps \in (0,1)$, in  $L^2_\sigma {(\R^d)} $ converging to  $ u^{E}_{0} $.
For every $\eps $ in $(0,1)$, we  consider $u^{\eps}$ a  weak Leray solution  associated with $u^{\eps}_{0}$.
Assume that there exists $c>0$ such that 
 \begin{equation}
\label{KatoCondition}  \eps  \int_{(0,T)}  \int_{ \mathcal{V}_{t,c\eps} }   |  \nabla u^{\eps}  |^{2} \,  dx dt \rightarrow  0 , \text{ when }\eps  \rightarrow  0 ,
\end{equation}
and that 
 \begin{equation}
\label{KatoConditionJump}    \int_{(0,T)}  \int_{ \mathcal{V}_{t,c\eps} \cap \Omega_{t,+} }   | [ \nabla u^{\eps}]  |^{2} \,  dx dt \rightarrow  0 , \text{ when }\eps  \rightarrow  0 .
\end{equation}
Then 
\begin{equation}
\label{Convergence}  \sup_{ (0,T)} \, \int_{\R^2} | u^{\eps}- u^{E} |^2 \, dx   \rightarrow  0  \text{ when }\eps  \rightarrow  0 ,
\end{equation}
\end{Theorem}
The proof of Theorem \ref{ThKato} is displayed in three parts corresponding respectively to Sections \ref{preuve1},  \ref{Preuve2} and \ref{Preuve3}. 
%
\section{A few comments}
\begin{itemize}
\item  Condition \eqref{KatoCondition}  is similar to Kato's original condition for the case of boundary layers attached to a fixed rigid wall, cf.  \cite{Tosio}. 
 On the other hand Condition \eqref{KatoConditionJump}, 
at first look, seems quite a strong extra assumption since there is no factor  $\eps  $ in front of the integral. 
However such an assumption is not that bad because of the regularizing effect of the Navier-Stokes equations. 
Indeed if the solution $u^{\eps}$ is smooth then 
 for every time $t$  the trace of  $[ \nabla u^{\eps}]$ on $\Sigma_t$ is well-defined, vanishes and Hardy's inequality can be applied, so that 
  \eqref{KatoConditionJump}  follows from the following condition: 
\begin{equation}
\label{CCC}   \eps^2  \int_{(0,T)}  \int_{ \mathcal{V}_{t,c\eps} }   |  \Delta u^{\eps}  |^{2} \,  dx dt   \rightarrow  0  \text{ when }\eps  \rightarrow  0 ;
\end{equation}
a condition which scales with the energy bound deduced from \eqref{NSBodyWeakEnergy} 
and Condition \eqref{KatoCondition}.
\ \par \ 

\item For shear flows,  the Navier-Stokes solutions have variations in $\varphi / \sqrt{\eps}$ so that Condition \eqref{KatoCondition} and Condition \eqref{KatoConditionJump}  are of course satisfied. 
 Indeed our  proof of Theorem \ref{ThKato}, following Kato's approach in  \cite{Tosio}, 
  involves a fake layer with  variations in $\varphi / \eps$, see \eqref{2D}  and \eqref{3D}. 
For more general flows, ansatz with  variations in $ \varphi / \sqrt{\eps}$ lead to Prandtl-type equations,  
see  \cite{Benedetto-Pulvirenti}  and  \cite{C-S}. 
 \ \par \ 

\item  In  \cite{sueur2012} we  consider the motion of a rigid body in an incompressible fluid occupying the complementary set in the space, with a no-slip condition at the interface, and we prove that a Kato type condition implies the convergence of both fluid  and body velocities. 
In this  paper we extend these results to the case of a vortex sheet,  that is to a fluid  interface with a more evolved dynamics. Indeed boundary layers associated with the no-slip condition on a fixed wall 
can be also viewed as vortex sheets with fixed support, see for instance \cite{CK,Kelliher-2017}. 
Loosely speaking Theorem \ref{ThKato} seems to indicate that despite that the dynamics of a vortex sheet is more subtle, the scale which is of interest for the inviscid limit is perhaps not worse than in the case of no-slip boundary layers. 
\ \par \ 

\item   In the case of the convergence of solutions  to  the incompressible Navier-Stokes solutions in a bounded domain with no-slip condition to smooth solutions to the incompressible Euler solution in the zero viscosity limit, 
  in addition to Kato's criterion, another criterion is given by Bardos and Titi in  \cite[Section 4.4]{BardosTiti-turbu}, see also \cite[Section 8 and 10] {Kelliher-2017}.  It involves only the behaviour of the  Navier-Stokes solutions on the boundary in the zero viscosity limit and the proof relies on 
  Kato's construction. 
  This criterion can be adapted to the present setting as a condition on the interface $\Sigma$ by substituting Lemma \ref{LemmaFake} 
  below instead of  Kato's construction.
  Indeed, by a direct energy estimate it is not difficult to see that the convergence  \eqref{Convergence}  holds if and only 
\begin{equation}
\label{BT-VS} 
\eps    \int_\Sigma   \det (\nabla \varphi,\curl  u^{\eps} ,[ u^E ]) \,  d\sigma  \rightarrow 0 , \text{ when }\eps  \rightarrow  0 ,
  \end{equation}
  where $\sigma$ is the surface measure on $\Sigma$. (In particular, to prove the direct part, we use that 
  $$ \eps \iint_{(0, T) \times \R^d} | \nabla u^{\eps} |^2 \,  dx dt  \rightarrow 0,  \text{ when }\eps  \rightarrow  0 ,$$
  as a consequence of \eqref{EulerBodyStrongEnergy}, \eqref{NSBodyWeakEnergy}, \eqref{Convergence} and of the convergence of the initial data).
  Moreover, under the assumptions of Theorem \ref{ThKato},
   for any   $\Psi$ in $C^\infty_0 (\Sigma ; \R^d) $ tangent to $\Sigma$, 
\begin{equation}
\label{BT-VS2} 
\eps  \int_\Sigma \curl  u^{\eps}  \cdot \Psi \rightarrow 0  \text{ when }\eps  \rightarrow  0   . 
  \end{equation}
 Since the proof of \eqref{BT-VS2}  can be easily adapted from  \cite[Section 4.4]{BardosTiti-turbu} and from the analysis performed  in the course of the proof of Theorem \ref{ThKato} (in particular by using  Lemma \ref{LemmaFake} below with   $\Psi$ instead of $[ u^E ]$ and following the treatment done in  Section \ref{Preuve3}  of the terms denoted by $R_{(iii)} $,    $R_{(iv)} $ and  $R_{(v)}$ in  Lemma \ref{lemmaGG}),
the details are left to the reader.
\ \par \ 

\item Theorem \ref{ThKato} proves that the conditions  \eqref{KatoCondition} and \eqref{KatoConditionJump} are sufficient for the convergence \eqref{Convergence}.  
The converse statement is an open question. 
Another open question is whether or not the convergence \eqref{Convergence} implies the interface condition  \eqref{BT-VS2}. 
To contrast with the classical setting, let us recall that Kato's condition and  Bardos and Titi's condition are proved to be sufficient  and necessary, respectively in  \cite{Tosio}  and \cite[Section 4.4]{BardosTiti-turbu}.
\ \par \ 

\item  As mentioned in Section \ref{sec-intro},   in the case of the convergence of solutions  to  the incompressible Navier-Stokes solutions in a bounded domain with no-slip condition to smooth solutions to the incompressible Euler solution in the zero viscosity limit,   some other variants of  Kato's criterion have been found, see 
 \cite{CEIV,CKV,Kelliher-2007,Kelliher-2008,Kelliher-2017,Wang}. Similar extensions of Theorem \ref{ThKato}
 can be obtained with minor modifications. 
 \ \par \ 
 
\item  We hope to extend our analysis to the case of compressible flows. Let us recall that  Coulombel and Secchi prove, in   \cite{cs} and \cite{cs-u}, the existence and uniqueness of  supersonic  compressible vortex sheets in two space dimensions. 
Therefore it would be interesting to obtain some sufficient conditions for the convergence of  solutions  to the compressible Navier-Stokes equations to these solutions in the zero-viscosity limit. 
In this direction let us mention that the case of  the inviscid limit of the compressible Navier-Stokes equations in a bounded domain, with the no-slip condition (and also in the case of the Navier slip-with-friction conditions), was tackled in \cite{sueur2014}.
\end{itemize}
%
\section{Energy estimate with an abstract corrector}
\label{preuve1}

Following Kato's approach, see \cite{Tosio}, we first observe that a corrector may help to deduce a $L^2$ stability estimate. 
For sake of notations, we  write temporarily $ u_{0} $ rather than  $ u^{\eps}_{0} $  and similarly, for the initial data,  $u$ rather than $ u^{\eps}$. 
Moreover the estimate involves an abstract corrector $v$ which will be chosen dependent on $\eps$ in the next sections. 
\begin{Lemma}
\label{lemmaGG}
If   $ v$  is such that 
 $u^E  + v  $  can be taken as a test function $\phi$ in \eqref{wNS} then 
\begin{gather} \nonumber
 \frac12 \|  u(t,\cdot) -  u^E  (t, \cdot)   \|_{L^2{(\R^d)}}^{2}   \leq  \frac12 \| u_{0} - u^E_{0}  \|_{L^2{(\R^d)}}^{2} 
 +  (u,   v )_{L^2 {(\R^d)}}  -  (u_{0},   v  |_{t=0})_{L^2{(\R^d)}}  
\\  \label{square2}  -  \int_{0}^{t}  R(s) ds       , 
\end{gather}
where
\begin{gather}
\label{FinalR}
 R(t) =    R_{(i)} (t) + \ldots +  R_{(v)} (t) ,
 \end{gather}
 with 
\begin{align*}
R_{(i)} &:=   \sum_\pm \,  \int_{\Omega_{t,\pm} }  ((u-u^E ) \cdot \nabla u^E  )   \cdot (u- u^E)   dx , \
 \\ R_{(ii)} &:= -  \eps    \sum_\pm \,    \int_{\Omega_{t,\pm}  }   \nabla u :  \nabla  u^E   dx , \
 \\  R_{(iii)} &:=  \sum_\pm \,  \int_{\Omega_{t,\pm} }  u \cdot  \big(  \partial_{t}  v    +   (u^E \cdot \nabla  v ) \big) dx , \ 
 \\  R_{(iv)} &:=   -   \sum_\pm \, \int_{\Omega_{t,\pm} }   v  \cdot   ((u-u^E) \cdot \nabla u ) dx , \, \
 \end{align*}
and
\begin{align*}  R_{(v)} &:=  -  \eps  \sum_\pm \,      \int_{\Omega_{t,\pm}  }   \nabla u :  \nabla v    dx .
\end{align*}
\end{Lemma}
\begin{proof}
For any $t \in [0,T ]$, we have, thanks to  \eqref{EulerBodyStrongEnergy} and \eqref{NSBodyWeakEnergy},
\begin{equation}
\label{square} 
 \|  u(t,\cdot) -  u^E  (t, \cdot)   \|_{L^2{(\R^d)}}^{2}   \leq \| u_{0}  \|_{L^2{(\R^d)}}^{2} + \|  u^E_{0}  \|_{L^2{(\R^d)}}^{2}  - 2 (u ,   u^E   )_{L^2{(\R^d)}}  (t) .
\end{equation}
By assumption, we can take  $\phi = u^E  + v  $ as a test function in \eqref{wNS}
 so that 
\begin{equation}
\label{WeakNSappl2}
(u,  u^E  + v )_{L^2{(\R^d)}}  (t) =  (u_{0},  u^E_{0}  + v  |_{t=0})_{L^2{(\R^d)}}     + \int_{0}^{t}  R(s) ds .
\end{equation}
where
\begin{gather}
\nonumber 
R(t) :=     \int_{\R^d}  u \cdot  \partial_{t}  ( u^E  + v )    dx  + 
   \int_{\R^d}  (u \cdot \nabla (u^E  + v)   \cdot u    dx
\\ \label{Monday} -  \eps      \int_{\R^d }   \nabla u :  \nabla(  u^E  + v )   dx   . 
\end{gather}

Combining \eqref{square}  with  \eqref{WeakNSappl2}  we obtain \eqref{square2} with $R(t) $ given by \eqref{Monday}.
Thus it only remains to prove that  \eqref{Monday} can be translated into \eqref{FinalR}. 
To do so let us split $R(t)$ into 
\begin{gather}
\label{decR}
R(t) = R_{E,+}  (t) + R_{E,-}  (t) + R_{F,+}  (t) + R_{F,-}  (t) ,
\end{gather}
 where 
\begin{gather*}
R_{E,\pm}  :=  \int_{\Omega_{t,\pm} }  u \cdot  \partial_{t}   u^E     dx + 
   \int_{\Omega_{t,\pm} }  (u \cdot \nabla u^E  )   \cdot u    dx
-  \eps      \int_{\Omega_{t,\pm}  }   \nabla u :  \nabla  u^E   dx   ,
\\  R_{F,\pm}  :=  \int_{\Omega_{t,\pm} }  u \cdot  \partial_{t}  v      dx  + 
   \int_{\Omega_{t,\pm} }  (u \cdot \nabla  v )  \cdot u    dx
-  \eps      \int_{\Omega_{t,\pm}  }   \nabla u :  \nabla v    dx  .
\end{gather*}
We first use that 
\begin{gather*}
  \int_{\Omega_{t,\pm} }  (u \cdot \nabla u^E  )   \cdot u   dx   = 
  \int_{\Omega_{t,\pm} }  (u^E \cdot \nabla u^E  )   \cdot u   dx  +  \int_{\Omega_{t,\pm} }  ((u-u^E ) \cdot \nabla u^E  )   \cdot u   dx
\end{gather*}
and the continuity of the normal component of   $ u -u^E $ at the interface to deduce that 
$$   \sum_\pm \,  \int_{\Omega_{t,\pm} }  ((u-u^E ) \cdot \nabla u^E  )   \cdot u^E dx = 0 .$$
Therefore 
\begin{gather}
\label{decompo} 
 \sum_\pm \,  \int_{\Omega_{t,\pm} }  (u \cdot \nabla u^E  )   \cdot u   dx   = 
  \sum_\pm \,  \int_{\Omega_{t,\pm} }  (u^E \cdot \nabla u^E  )   \cdot u   dx 
  \\  \nonumber +   \sum_\pm \,  \int_{\Omega_{t,\pm} }  ((u-u^E ) \cdot \nabla u^E  )   \cdot (u- u^E)   dx
\end{gather}
Moreover, since $u^E $ satisfies  the Euler equation in a strong sense in  both $\Omega_{t,\pm} $ and $u$ is divergence free and continuous at the interface, we obtain, upon an integration by parts, the following identity: 
\begin{gather}
\label{deo} 
 \sum_\pm \,  \int_{\Omega_{t,\pm} }  u \cdot  \partial_{t}   u^E    dx  = -   \sum_\pm \,  \int_{\Omega_{t,\pm} }  (u^E \cdot \nabla u^E  )   \cdot u   dx   .
 \end{gather}
Adding $R_{E,+}$ and $R_{E,-}$ and using 
 \eqref{decompo} and  \eqref{deo}   we arrive at 
\begin{gather}
\label{pourE} 
  \sum_\pm \,  R_{E,\pm}   =
     \sum_\pm \,  \int_{\Omega_{t,\pm} }  ((u-u^E ) \cdot \nabla u^E  )   \cdot (u- u^E)   dx
-  \eps      \int_{\Omega_{t,\pm}  }   \nabla u :  \nabla  u^E   dx   .
\end{gather}

On the other hand, 
\begin{gather}
\nonumber R_{F,\pm}  =  \int_{\Omega_{t,\pm} }  u \cdot  \big(  \partial_{t}  v    +   (u^E \cdot \nabla  v ) \big)    dx
 +  \int_{\Omega_{t,\pm} }  u \cdot   ((u-u^E) \cdot \nabla  v )    dx
\\ \label{pourF} -  \eps      \int_{\Omega_{t,\pm}  }   \nabla u :  \nabla v    dx  .
\end{gather}

Moreover using once again the continuity of  the normal component of   $u -u^E $ at the interface we arrive at 
\begin{gather}
\label{pourFsi} 
  \sum_\pm \, \int_{\Omega_{t,\pm} }  u \cdot   ((u-u^E) \cdot \nabla  v )    dx 
 = -   \sum_\pm \, \int_{\Omega_{t,\pm} }   v  \cdot   ((u-u^E) \cdot \nabla u )    dx .
\end{gather}

 Thus combining
   \eqref{decR}, \eqref{pourE}, \eqref{pourF} and \eqref{pourFsi}  
 we arrive at \eqref{FinalR} and the proof of Lemma \ref{lemmaGG} is completed. 
 \end{proof}

\section{Construction of an almost odd transition layer}
\label{Preuve2} 

In the following result, we make use of the Landau notations $o(1)$ and $O(1)$ for quantities respectively converging to $0$ and bounded with respect to the limit $\eps \rightarrow 0^{+}$.
Let $c>0$  such that \eqref{KatoCondition}  and  \eqref{KatoConditionJump} are satisfied. 
Recall, for a function $f$,  the notation  $\tilde{f}$ in the beginning of Section \ref{sec-MR}.
\begin{Lemma}
\label{LemmaFake}
There exists a family $(v^\eps)$, indexed by $\eps $ in $(0,1)$, in  $C ( [0,T ]  ;  L^2_\sigma {(\R^d)})$ with the following properties: for every  $\eps $ in $(0,1)$, 
\begin{gather}
 \label{support}
 \text{for every } t \in  [0,T ], \quad \text{supp } v^\eps (t,\cdot) \subset  \mathcal{V}_{t,c\eps} , 
 \\  \label{Fake7}
u^E  + v^\eps  \in C ( [0,T ]  ;  L^2_\sigma{(\R^d)}) \cap  L^2 ( [0,T ] ;  H^1 {(\R^d)} ) ,
\end{gather}
such that
\begin{gather}
 \label{Fake0}
v^\eps = O( 1 ) \text{ in }   L^{\infty} ( [0,T ]  \times \R^{3} ),
\\ \label{Fake1}
v^\eps = O( \varepsilon^{\frac{1}{2}}) \text{ in }C ( [0,T ]  ;  L^2{(\R^d)}),
\\   \label{FakeImp}
\varphi v^\eps = O( \varepsilon ) \text{ in }   L^{\infty} ( [0,T ]  \times \R^{3} ),
\\   \label{Fake4}
 \sup_{t\in (0,T)} \, \| \nabla v^\eps \|_{L^{2} ( \mathcal{V}_{t,c\eps} \cap \Omega_{t,+} ) } = O( \varepsilon^{-\frac{1}{2}}  ) ,
\end{gather}
and
\begin{gather}
\label{symm}
v^\eps + \widetilde{v^\eps }  = O( \varepsilon ) \text{ in }   L^{\infty} ( [0,T ]  \times \R^{3} ), 
\\ \label{Fake2}
\partial_{t} v^\eps + u^E \cdot \nabla  v^\eps  = O( \varepsilon^{\frac{1}{2}}) \text{ in }C ( [0,T ]  ;  L^2{(\R^d)}),
\end{gather}
\end{Lemma}
\begin{Remark}
Above we have written separately the estimates \eqref{Fake0}-\eqref{Fake4} which are similar to the ones in Kato's original paper  \cite{Tosio} and the estimates \eqref{symm} and \eqref{Fake2} which are  two new requirements  useful in the case of vortex sheets. 
\end{Remark}
\begin{proof}[Proof of  Lemma \ref{LemmaFake}] 
Let  $\xi: [ 0,+\infty) \rightarrow [ 0,+\infty) $ be a smooth cut-off function such that $\xi(0) = 0$,  $\xi'(0) = 1$ and  $\xi(r) = 0$ for $r \geq c$. 
Recall the notation $[\cdot] $ in \eqref{reflexion}. 
Set, for $t$ in $ [0,T ]$,  $x$ in $ \Omega_{t,\pm} $ and $\eps $ in $(0,1)$, 
\begin{gather}
\label{2D}  v^\eps  := -  \nabla^\perp \big(  \frac{\varepsilon }{2} \xi(\pm \frac{\varphi }{\eps} ) [ u^E \cdot  \nabla^\perp \varphi  ]\vert_{\varphi = 0}  \big)  \quad   \text{ if } d=2, 
\\ \label{3D}  v^\eps  :=  \curl \big(  \frac{\varepsilon }{2}  \xi(\pm \frac{\varphi }{\eps} ) [ u^E  ]\vert_{\varphi = 0} \times \nabla \varphi \big) \quad   \text{ if } d=3.
 \end{gather}
Then we easily check that the family $(v^\eps)_{\eps \in (0,1)}$ is in  $C ( [0,T ]  ;  L^2_\sigma {(\R^d)})$  and satisfies  \eqref{support}-\eqref{Fake2}. (Observe  in particular that we use \eqref{eikonal} to obtain \eqref{Fake2}).
\end{proof}
Let us display here a remark which could be useful to get an insight of the whole strategy, with some anticipation on the rest of 
 the proof of Theorem \ref{ThKato}. 
%
\begin{Remark}
Observe that above the $v^\eps$ are constructed such that 
  $[ u^E + v^\eps  ] = 0$ on $\Sigma$ but without any condition on  $[ \nabla ( u^E  + v^\eps )  ]$ despite that  it is expected that a  nice physical approximation  $u^\eps_a $  of $u^\eps $ should satisfy $[ \nabla u^\eps_a  ]=0$ on $\Sigma$. 
However in the next section we will combine Lemma \ref{lemmaGG} and  Lemma \ref{LemmaFake} and we will estimate  the right hand side of  
  \eqref{square2} without any further integration by parts of the diffusive terms $R_{(ii)} $ and $ R_{(v)}$ so that the lack of information regarding   $[ \nabla ( u^E  + v^\eps )  ]$ at the interface will not be a problem. 
 We therefore spare  a degree of freedom which is used in  Lemma \ref{LemmaFake} to insure the almost oddness of the transition layer stated in \eqref{symm}. 
  Such a condition has no reason to be physical but  will be crucial in the treatment of the convective term  $R_{(iv)} $ in Section \ref{Preuve3}. 
\end{Remark}
\section{End of the proof of Theorem \ref{ThKato}}
\label{Preuve3} 

We now go back to the proof of  Theorem \ref{ThKato}. 
We apply Lemma \ref{lemmaGG} with $v^\eps$ instead of $v$ 
  where  $(v^\eps)$ is a family as in Lemma \ref{LemmaFake}. Indeed a density argument, \eqref{Fake7} and the piecewise smoothness of $v^\eps$ allows us to take 
  $u^\eps + v^\eps$ as a test function $\varphi$ in \eqref{wNS}. 
(We now stop dropping  the index ${\eps}$ of 
  $ u^{\eps}_{0} $  and $ u^{\eps}$ but we will keep the notations 
 $ R_{(i)}$, \ldots,  $R_{(v)} $, without any extra index, being understood that these terms depend on ${\eps}$).
 We are now going to bound the various terms in the right hand side of  \eqref{square2}. 
 
 Since the Navier-Stokes  initial data  $u^\eps _{0} $ converges to the Euler one in $L^2 {(\R^d)} $ as the viscosity  $\eps$ goes to $0$, it is bounded, and so is the corresponding 
  Navier-Stokes solution $u^\eps $ for almost every time, according to the energy estimate  \eqref{NSBodyWeakEnergy}. 
  Therefore, by Cauchy-Schwarz' inequality and  \eqref{Fake1},  for almost every time,
\begin{gather}
\label{titrucs}
\vert   (u^\eps ,   v^\eps )_{L^2{(\R^d)}}  \vert  \leq C \eps^{\frac{1}{2}}    \text{ and }  \vert  (u^\eps _{0},   v^\eps  |_{t=0})_{L^2{(\R^d)}} \vert  \leq C \eps^{\frac{1}{2}}  .
 \end{gather}
Let us warn the reader that we will use the same notation $C$ for various constants which may change from line to line, but always independent of $ \eps$. 

Using that the Euler solution is piecewise smooth, we arrive at 
\begin{gather}
\label{esti(i)}
 \vert   R_{(i)} (t)  \vert  
  \leq  C   \|  u^\eps (t,\cdot) -  u^E  (t, \cdot)   \|_{L^2{(\R^d)}}^{2}  .
 \end{gather}
By Cauchy-Schwarz' inequality, 
\begin{gather}
\label{esti(ii)}
   \vert   R_{(ii)}  \vert    \leq   C  \eps    \|    \nabla u^\eps    \|_{L^2{(\R^d)}}  
\end{gather}
Using again Cauchy-Schwarz' inequality and  \eqref{Fake2}, we arrive at 
\begin{gather}
   \label{esti(iii)}
    \vert   R_{(iii)}  \vert    \leq   C \varepsilon^{\frac{1}{2}} \|   u^\eps   \|_{L^2 {(\R^d)}}  . 
    \end{gather}
   Let us continue with estimating  $R_{(v)} (t) $, keeping the best for the end. 
Using again Cauchy-Schwarz' inequality,   \eqref{support}  and  \eqref{Fake2}, we arrive at 
\begin{gather}
   \label{esti(v)}
     \vert   R_{(v)} (t)  \vert    \leq C \eps^{\frac{1}{2}}   \|      \nabla u^\eps    \|_{L^2 {( \mathcal{V}_{t,c\eps} )}} .
\end{gather}

It remains to deal with $R_{(iv)} (t) $. This is where the treatment is quite different from the one performed in the traditional setting of boundary layers along an impermeable wall. 
By a change of variable (observe that for $\eps$ small enough, the reflexion across $\Sigma_t$, introduced at the beginning of Section \ref{sec-MR}, is well-defined on the support of $v^\eps$)  and  \eqref{symm}, 
\begin{align*}
 R_{(iv)} &=   R_{(iv),a}  + R_{(iv),b} + R_{(iv),c} +R_{(iv),d} ,
   \end{align*}
with 
\begin{gather*}
 R_{(iv),a} :=   -  \, \int_{\Omega_{t,+} }   v^\eps  \cdot (   [ u^\eps ]   \cdot \nabla u^\eps  )    dx, \quad  
\\   R_{(iv),b} :=   \, \int_{\Omega_{t,+} }   v^\eps  \cdot ( [u^E]  \cdot \nabla u^\eps   )    dx , \quad
  \\   R_{(iv),c} :=  - \, \int_{\Omega_{t,+} }   v^\eps  \cdot ((\widetilde{u^\eps }-\widetilde{u^E}) \cdot [ \nabla u^\eps ] )    dx ,
   \end{gather*}
 and
\begin{gather*}
  R_{(iv),d} := - \, \int_{\Omega_{t,+} }  ( v^\eps + \widetilde{v^\eps} ) \cdot ((\widetilde{u^\eps }-\widetilde{u^E}) \cdot \widetilde{ \nabla u^\eps } )    dx .
 \end{gather*}
\begin{itemize}
\item  By   \eqref{support},  \eqref{FakeImp} and  Cauchy-Schwarz' inequality,
\begin{gather*}
   \vert   R_{(iv),a} (t)  \vert    \leq    \varepsilon    \| \varphi^{-1} \,  [ u^\eps  (t,\cdot)]    \|_{L^2{(\mathcal{V}_{t,c\eps} \cap\Omega_{t,+} )}} \,  
     \|  \nabla u^\eps    \|_{L^2{(\mathcal{V}_{t,c\eps} \cap\Omega_{t,+}  )}} \,    
 \end{gather*}
Moreover by Hardy's inequality, 
$$ \| \varphi^{-1} \,  [ u^\eps (t,\cdot)]    \|_{L^2{(\mathcal{V}_{t,c\eps} \cap\Omega_{t,+}  )}}  \leq C  \| \nabla  u^\eps   \|_{L^2{( \mathcal{V}_{t,c\eps} )}}  ,$$
so that 
\begin{gather}
   \label{esti(iva)}
   \vert   R_{(iv),a} (t)  \vert    \leq  C  \varepsilon        \|  \nabla u^\eps    \|^2_{L^2{(\mathcal{V}_{t,c\eps}   )}}      .
 \end{gather}
\item  By  \eqref{support}, \eqref{Fake1} and Cauchy-Schwarz' inequality, 
\begin{gather}
   \label{esti(ivb)}
   \vert   R_{(iv),b} (t)  \vert    \leq  C  \varepsilon^{\frac{1}{2}}        \|  \nabla u^\eps    \|_{L^2{(\mathcal{V}_{t,c\eps} )}} .
 \end{gather}
\item  By  \eqref{support}, \eqref{Fake0} and Cauchy-Schwarz' inequality, 
\begin{gather}
   \label{esti(ivc)}
   \vert   R_{(iv),(c)} (t)  \vert    \leq  C  \| u^\eps -u^E  \|_{L^2{(\R^d)}}  \|  [ \nabla u^\eps ]  \|_{L^2{( \mathcal{V}_{t,c\eps} \cap\Omega_{t,+})}} .
 \end{gather}
\item  Finally, by \eqref{support}, \eqref{symm} and Cauchy-Schwarz' inequality,  
\begin{gather}
   \label{esti(ivd)}
   \vert   R_{(iv),(d)} (t)  \vert    \leq  C  \varepsilon   \| u^\eps -u^E  \|_{L^2{(\R^d)}}   \|   \nabla u^\eps   \|_{L^2{( \mathcal{V}_{t,c\eps} })} .
 \end{gather}
\end{itemize}
Gathering   \eqref{esti(iva)}, \eqref{esti(ivb)}, \eqref{esti(ivc)}  and \eqref{esti(ivd)}  we arrive at 
\begin{gather}
   \label{esti(iv)}
   \vert   R_{(iv)}  \vert    \leq   C  \| u^\eps -u^E  \|^2_{L^2{(\R^d)}}  + C  \varepsilon       \|  \nabla u^\eps    \|^2_{L^2{(\mathcal{V}_{t,c\eps}   )}}  
\\ \nonumber    +  C  \varepsilon^{\frac{1}{2}}        \|  \nabla u^\eps    \|_{L^2{(\mathcal{V}_{t,c\eps} )}} + C \|  [ \nabla u^\eps ]  \|^2_{L^2{( \mathcal{V}_{t,c\eps} \cap\Omega_{t,+})}}   .
  \end{gather}

Using \eqref{titrucs}, \eqref{esti(i)}, \eqref{esti(ii)},  \eqref{esti(iii)},  \eqref{esti(v)} and  \eqref{esti(iv)}
to bound the various terms in the right-hand side of \eqref{square2}
we arrive at 
\begin{gather} \nonumber
 \|  u^\eps (t,\cdot) -  u^E  (t, \cdot)   \|_{L^2{(\R^d)}}^{2}   \leq \| u^\eps _{0} - u^E_{0}  \|_{L^2{(\R^d)}}^{2}  + C    \int_{0}^{t}   \| u^\eps -u^E  \|^2_{L^2{(\R^d)}} ds  
 \\  \nonumber +   C \eps^{\frac{1}{2}}   +   C  \varepsilon^{\frac{1}{2}} \int_{0}^{t}    \|   u^\eps   \|_{L^2 {(\R^d)}}  ds 
 + C  \eps  \int_{0}^{t}  \|    \nabla u^\eps    \|_{L^2{(\R^d)}}  ds
 \\  \label{GronW}   + C \int_{0}^{t}    \Big(    \eps^{\frac{1}{2}}   \|      \nabla u^\eps    \|_{L^2 {( \mathcal{V}_{t,c\eps} )}} 
 + \varepsilon       \|  \nabla u^\eps    \|^2_{L^2{(\mathcal{V}_{t,c\eps} )}} 
+  \|  [ \nabla u^\eps ]  \|^2_{L^2{( \mathcal{V}_{t,c\eps} \cap\Omega_{t,+})}}  \Big) ds       .
\end{gather}
Thanks to \eqref{NSBodyWeakEnergy} and Cauchy-Schwarz' inequality, the terms in the second line above converge to $0$ as the viscosity  $\eps$ goes to $0$. 
The terms in the third line above also converge to $0$ thanks to Conditions \eqref{KatoCondition} and \eqref{KatoConditionJump}.
Therefore, by Gronwall's lemma, we get the convergence stated in  \eqref{Convergence} and the proof of Theorem \ref{ThKato} is completed.

%

\section*{Acknowledgements}

The author  thanks  the Agence Nationale de la Recherche, Project DYFICOLTI, grant ANR-13-BS01-0003-01, Project IFSMACS, grant ANR-15-CE40-0010  and Project  BORDS, grant  ANR-16-CE40-0027-01.

%
%

\end{document}